\documentclass[11pt]{article}
\usepackage{amsmath}
\usepackage{amsfonts}
\usepackage{amsthm} 
\usepackage{amssymb}  
\usepackage{latexsym}
\usepackage{euscript}
\usepackage{txfonts}
\newcommand{\msf}[1]{{\mathsf {#1}}}

\newcommand{\mcal}[1]{{\mathcal {#1}}} 

 \newtheorem{theorem}{Theorem} 
 \newtheorem{lemma}[theorem]{Lemma}

\newtheorem{proposition}[theorem]{Proposition}

\theoremstyle{definition}
\theoremstyle{definition}\newtheorem{definition}[theorem]{Definition}
\theoremstyle{definition}
\theoremstyle{definition}
\theoremstyle{definition}
\theoremstyle{definition}
\theoremstyle{plain}\newtheorem*{lemma*}{Lemma}
\theoremstyle{plain}\newtheorem*{theorem*}{Theorem}
\theoremstyle{plain}\newtheorem*{proposition*}{Proposition}
\theoremstyle{plain}\newtheorem*{corollary*}{Corollary}
\theoremstyle{remark}\newtheorem*{remark*}{Remark}
\theoremstyle{remark}\newtheorem*{remarks*}{Remarks}
\theoremstyle{definition}\newtheorem*{conjecture*}{Conjecture}

\theoremstyle{definition}\newtheorem*{definition*}{Definition}
\theoremstyle{definition}\newtheorem*{definitions*}{Definitions}

\theoremstyle{definition}\newtheorem*{example*}{Example}
\theoremstyle{definition}\newtheorem*{question*}{Question}
\theoremstyle{definition}\newtheorem*{questions*}{Questions}
\theoremstyle{definition}\newtheorem*{hypothesis*}{Hypothesis}

\def\qed{\ifhmode\unskip\nobreak\fi\ifmmode\ifinner\else\hskip5pt\fi\fi

\hfill\hbox{\hskip5pt\vrule width4pt height6pt depth1.5pt\hskip1pt}}

\newcommand{\ov}[1]{\ensuremath{\overline{#1}}} 

\newcommand{\dimply}{\ensuremath{\:\Longleftrightarrow\:}}
\newcommand{\RR}{\ensuremath{{\mathbb R}}}     %bb R
\newcommand{\R}[1]{\ensuremath{{\mathbb R}^{#1}}} %bb R superscripted 
\newcommand{\QQ}{\ensuremath{{\mathbb Q}}}     %bb Q
\newcommand{\ZZ}{\ensuremath{{\mathbb Z}}}	   %bb Z
\newcommand{\NN}{\ensuremath{{\mathbb N}}}	   %bb N   
\newcommand{\Np}{\ensuremath{{\mathbb N}_{+}}}  %bb N^plus
\newcommand{\Qp}{\ensuremath{{\mathbb Q}_+}}	   %bb Q_plus   
\newcommand{\co}{\colon\thinspace} %% Colon with correct spacing for maps
\def\empty{\varnothing}
\newcommand{\sm}[1]{\ensuremath{\setminus #1}}

\newcommand{\p}{\ensuremath{\partial}}
  \def\mylabel#1{\label{#1}}
\reversemarginpar
\renewcommand{\v}{\varphi}

 \def\P{{\mcal P}}
 \def\F{{\mcal F}}
 \def\E{{\mcal E}}

\begin{document}
\title{\LARGE \bf  Monotone flows with dense periodic orbits} 
\author{\Large Morris W.\ Hirsch \thanks{I thank {\sc Eric Bach,
    Bas Lemmens}  and {\sc Janusz Mierczy\'nski} for helpful
    discussions.}   \\University of Wisconsin}

\maketitle
%% \begin{flushright} {\footnotesize MONOTONE/JDE/New-revision-4.tex \normalsize}
%% \end{flushright}

%% \fbox{ADD new theorem on polyhedral cones.}

 \begin{abstract} 
  Give $\R n$ the (partial) order $\succeq$ determined by a closed
  convex cone $K$  having nonempty interior: $y\succeq x \dimply y-x \in K$.
  Let  $X\subset\R n$ be a connected open set and $\v$  a  flow on
  $X$ that is monotone for this order: if $y\succeq x$ and $t\ge 0$,
  then $\v^ty\succeq \v^t y$.

   {\em Theorem.} If  periodic points of $\v$ are dense in $X$,  then 
   $\v^t$ is the identity map of $X$ for  some $t \ge 0$.

  %%  {\bf ***[Add applications: analytic, $C^1$, ODEs?] }

   \end{abstract}
 
 \tableofcontents

%------------------------------------------------------------
 \section{Introduction}   \mylabel{sec:intro}
 %-----------------------------------------------------------

 Many dynamical systems, especially those used as models in applied
 fields, are {\em monotone}: the state space has an order relation
 that is preserved in positive time. A recurrent theme is that bounded
 orbits tend toward periodic orbits. For a sampling of the large
 literature on monotone dynamics, consult the following works and the
 references therein: 
{\small \cite{AngeliHirschSontag, Anguelov12, BenaimHirsch99a,
    BenaimHirsch99b, DeLeenheer17, Dirr15, Elaydi17, EncisoSontag06,
    Grossberg78, Hess-Polacik93, HS05, Kamke32, LajmanovichYorke,
    Landsberg96, LeonardMay75, Matano86, Mier94, Potsche15,
    RedhefWalter86, Selgrade80, Smale76, Smith95, Smith17, Volkmann72,
    Wang17, Walcher01}.}

%%
%% {\small \cite{Akian03, Anguelov12, BenaimHirsch99a, BenaimHirsch99b,
%%     DeLeenheer17, Dancer-Hess91, Dirr15, Elaydi17, EncisoSontag06,
%%     Grossberg78, Hess-Polacik93, HS05, Kamke32, LajmanovichYorke,
%%     Landsberg96, LeonardMay75, Luis11, Matano84, Matano86, Mier94,
%%     Potsche15, Selgrade80, Smale76, Smith95, Smith17, Volkmann72,
%%     Walcher01}.}

Another common dynamical property is {\em dense periodicity}: periodic
points are dense in the state space.  Often considered typical of
chaotic dynamics, this condition is closely connected to many other
important dynamical topics, such as structural stability, ergodic
theory, Hamiltonian mechanics, smoothness of flows and
diffeomorphisms. 

The main result in this article, Theorem \ref{th:main}, is that a large class of
monotone, densely periodic flows $\v:=\{\v^t\}_{t\in\RR}$ on open
subsets of $\R n$ are
(globally) {\em periodic}: There exists $t>0$ such that $\v^t$ is the
identity map.

\subsection{Terminology}
Throughout this paper $X$ and $Y$ denote (topological) spaces.
The closure of a subset of a space  is   $\ov S$.  
 The group of homeomorphisms of $X$ is  $\mcal H (X)$.
 
$\ZZ$ denotes the integers, $\NN$ the nonnegative integers, and $\Np$
 the positive integers. $\RR$ denotes the reals, $\QQ$ the rationals,
 and $\Qp$ the positive rationals. 
$\R n$ is Euclidean $n$-space.

Every manifold $M$ is assumed  metrizable with empty boundary $\p
M$, unless otherwise indicated.

Every map is assumed continuous unless otherwise described.  $f\co
X\approx Y$ means $f$ maps $X$ homeomorphically onto $Y$.  The
identity map of $X$ is $id_X$.

If $f$ and $g$ denote maps, the composition of $g$ following $f$  is the map
$g\circ f\co x\mapsto g(f(x))$ (which may be empty).
 We set $f^0:=id_X$, and recursivelydefine the 
 $k$'th iterate $f^k$ of $f$ as $f^k:=f\circ f^{k-1}, \ (k \in \Np)$.

The {\em orbit} of $p$ under $f$ is the set
\[\mcal O(p):=\big\{f^k (p)\co k\in \NN\big\},
\]
denoted as $\mcal O (p,f)$ to record $f$.

When $f (X) \subset X$, the orbit of $P\subset X$ is defined as
\[  
\mcal O (P)= \mcal O (P,f) := \bigcup_{p\in P}\mcal O (p).
\]
A set $A\subset X$ is {\em invariant} under $f$ if $f(A)=A =f^{-1}
(A)$.  The {\em fixed set} and {\em periodic set} of $f$ are the
respective invariant sets
\[
 \F(f) :=\big\{x\co f(x)=x\big\}, \qquad \P(f):= \bigcup_{n\in \Np}
 \F(f^n).
\]

\paragraph{Flows} 
A {\em flow} $\psi:=\{\psi^t\}_{t\in \RR}$ on $Y$, denoted formally by
$(\psi, Y)$, is a continuous action of the group $\RR$ on $Y$: a
family of homeomorphisms $\psi^t\co Y\approx Y$ such that the map
\[
\RR \to \mcal H (Y), \quad t\mapsto \psi^t
\]
is a homomorphism, and the {\em evaluation map}
 %============================================================
  \begin{equation}\label{eq:ev}
%============================================================
 \msf {ev}_\psi\co\RR \times Y \to Y, \qquad (t,x) \mapsto \psi^tx
\end{equation}
  %------------------------------------------------------------
  is   continuous.  

A set $A\subset Y$ is  {\em invariant under $\psi$}  if it is invariant under
every $\psi^t$.  When this holds,  the {\em restricted flow}
$(\psi \big |A)$ on $A$  has the evaluation map
$\msf {ev}_\psi\big|\,\RR \times A$.

The {\em orbit} of $y\in Y$ under $\psi$ is the invariant set
\[
   \mcal O (y):=\big \{\psi^t y\co t\in \RR \big\},
   \]
denoted formally  by $\mcal O (y, \psi)$.  
Orbits of distinct points
either coincide or are disjoint.

The 
{\em periodic} and {\em equilibrium} sets of $\psi$ are the respective
invariant
sets
\[
   \P(\psi):=\bigcup_{t\in\RR}\F(\psi^t), \qquad
   \E(\psi):= \bigcap_{t\in\RR} \F(\psi^t),
   \]
denoted by $\P$ and $\E$ if  $\psi$ is clear from the context.
Points in $\E$ are called {\em stationary}.  The orbit of a
nonstationary periodic point is a {\em cycle}. 

The {\em period} of $p\in Y$ is   $r=\msf{per}(p)>0$, provided
\[p\in \P\sm \E, \quad r=\min\big\{t>0\co \psi^tp=p\big\}  
\]
The set of  {\em $r$-periodic} points is
\[\P^r (\Psi) =\P^r:= \big\{p\in \P\co\msf{per} (p)=r\big\}\]

When $p$ is $r$-periodic, the restricted flow $\psi \big|O(p)$ is
 topologically conjugate to the rotational flow on the topological
 circle $\RR/r\ZZ$,  covered by the translational flow on $\RR$
 whose evaluation map is $(t,x) \mapsto t+x$.
 
The flow $(\psi, Y)$ is:
\begin{itemize} 
  
\item {\em trivial} if $\E =Y$,

\item {\em densely periodic} if $\P $ is dense in $Y$,

\item {\em pointwise periodic} if $\P  = Y$,

\item {\em periodic} if there exists  $l > 0$ such that  $\psi^l =id_Y$.
  Equivalently: the homomorphism $\psi \co \RR \to \mcal H (Y)$
  factors through a homomomorphism of the circle group $\RR/l\ZZ$. 
\end{itemize}

\paragraph{Order} Recall that a  (partial)  {\em order} on $Y$ is a
binary relation $\preceq$ on $Y$ that is reflexive, transitive and
antisymmetric.
In this paper all orders are {\em closed}:
the set $\{(u,v)\in Y\times Y\co u\preceq v\}$ is closed.
  The  {\em trivial order} is:   $x\preceq y \dimply x=y$.

 The pair
$(Y,\preceq)$ is an {\em ordered space}, denoted by $Y$ if 
the order is clear from the context.

We write $y\succeq x$ to mean $x\preceq y$.  If $x\preceq y$ and $x\ne
y$, we write $x\prec y$ and $ y \succ x$.
  For sets $A, B\subset Y$, the condition $A\preceq B$ means $a\preceq
  b$ for all $a\in A, \,b\in B$, and similarly for the relations
  $\prec,\, \succeq, \,\succ$.

The {\em closed order interval} spanned by $a, b \in Y$ is the closed
set
\[ [a,b]:=\{y\in Y\co a \preceq y\preceq b\},\]
and its interior is the {\em open order interval} $[[a,b]]$.  We write
$a\ll b$ to indicate $[[a,b]]\ne\varnothing$.

An {\em order cone}
 $K\subset \R n$ is  a closed convex cone that is
pointed (contains no affine line) and  solid (has nonempty
interior).  $K$ is {\em polyhedral} if is the intersection of finitely
many closed linear halfspaces. 

The {\em $K$-order}, denoted by $\preceq_K$ for clarity, is
the order defined  on every subset of $\R n$ by 
\[
 x \preceq_K y \dimply y-x\in K.
\]
For the $K$-order on
$\R n$,   every  order interval is a convex $n$-cell.

 A map $f\co X \to Y$ between ordered spaces is {\em monotone} when
 \[x\preceq_X x' \implies f(x) \preceq_Y f (x').
 \]
 When $Y$ is ordered and $g\co Y'\to Y$ is injective, $Y'$ has a unique
 {\em induced order} making $g$ monotone.

Every subspace $S\subset Y$ is given the order induced by the
inclusion map $S\hookrightarrow Y$.  When this order is trivial, $S$
is {\em unordered}.

 A flow $\psi$ on an ordered space is {\em monotone} iff the maps
 $\psi^t$ are monotone.
It is easy to see that  every cycle for a monotone flow is unordered.

This is the chief result:
%============================================================
\begin{theorem}\mylabel{th:main}
%============================================================
Assume:  
  \begin{description}
 
  \item[(H1)] $K\subset \R n$ is an order cone.

  \item[(H2)] $X\subset \R n$ is open and connected.

  \item[(H3)] $X$ is $K$-ordered.

  \item[(H4)] The flow $(\v, X)$ is monotone and densely periodic.

\end{description}
  Then $\v$ is  periodic.
\end{theorem}
%------------------------------------------------------------
\noindent The proof will be given after results about various types of
flows.

%------------------------------------------------------------
\section{Resonance}   \mylabel{sec:resonance}
%------------------------------------------------------------

In this section:
\begin{itemize}
\item $Y$ is an ordered space,

\item $(\psi,Y)$ is a monotone flow,

\item
  $ \mcal O (y)$ is the orbit of $y$ under $\psi$,

\item $\mcal O (y, \psi^t)$  is the orbit of $y$ under  $\psi^t$.
\end{itemize}
%======================================================
 \begin{lemma}\mylabel{th:colimit}
%============================================================
Assume:
\begin{description}

   \item[(i)] \mbox{$\big\{p_k\big\}$ and $ \big\{q_k\big\}$ are seqences in $Y$ converging to
     $y\in Y$},

   \item[(ii)] $\mcal O (p_k) \prec \mcal O(q_k), \quad (k \ge 1).$

    \end{description}
Then  $y\in \E$.
 \end{lemma}
%---------------------------------------------------------------
\begin{proof} By (i) and  continuity of $\psi$ we have
%============================================================
  \begin{equation}\label{eq:limvt}
%============================================================
 \lim_{k\to \infty}\psi^t p_k =   \lim_{k\to \infty}\psi^t q_k =\psi^t y, \qquad (t\in \RR)
 \end{equation}
%------------------------------------------------------------
while (ii) and monotonicity of $\psi$ imply
%============================================================
   \begin{equation}\label{eq:pktqk}
%============================================================
k\in \NN, \  t\in \RR \implies  p_k \preceq \psi^t q_k.
\end{equation}
%------------------------------------------------------------
   From  (i),  (\ref{eq:limvt}). (\ref{eq:pktqk}) and continuity of $\psi$ we infer:
\[ y\preceq \psi^t y,\qquad (t\in \RR). \]
Therefore  monotonicity shows that $\psi^ty=y$ for all $t\in \RR$.
\end{proof}

\paragraph{Rationality}  Rational numbers play a surprising role in monotone flows: 
%%%%%%%%%%%%%%%%%%%%%%%%%%%%%%%%%%%%%%%%%%%%%%%%%%%%%%%%%%%%%
 \begin{theorem}        \mylabel{th:poa}
%%%%%%%%%%%%%%%%%%%%%%%%%%%%%%%%%%%%%%%%%%%%%%%%%%%%%%%%%%%%%%
 Assume $r,s >0$.  If
\[ p \in \mcal P^r, \quad  q \in \P^s, \quad p \prec  q,
\quad \text{and} \quad  \mcal O (p) \not \prec\mcal O (q),
\]
then $r/s$ is rational.
\end{theorem}
%------------------------------------------------------------
 \begin{proof}
We have
 %============================================================
 \begin{equation}                \label{eq:gama}
 %============================================================
   \mcal  O (p) \cap \mcal O (q)=\empty
 \end{equation}
 %------------------------------------------------------------
 because $p \prec q$
 and every cycle in a monotone flow us unordered.
 
Note that  the restriction of $\psi^s$ to the  circle $\mcal O(p)$ 
is conjugate to the rotation $\msf R_{2\pi r/s}$ of the unit circle $ \msf
C \subset\R 2$ through  $2\pi r/s$ radians.

Arguing by contradiction, we provisionally assume  $r/s$ is irrational.
Then the orbit of $\msf R_{2\pi r/s}$ is dense in $\msf
C$,\footnote{This was discovered--- in the 14th century!---  by  {\sc Nicole
    Oresme}. See {\sc Grant} \cite{Grant71}, {\sc
    Kar} \cite {Kar03}.  A short proof based on the pigeon-hole
  principle is due to {\sc Speyer} \cite{Speyer17}.  Stronger density
  theorems are given in {\sc Bohr} \cite{Bohr23}, {\sc Kronecker}
  \cite{Kronecker}, and {\sc Weyl} \cite{Weyl10, Weyl16}.}
and the cojugacy described above implies
%============================================================
\begin{equation}\label{eq:oppr}
%============================================================
\mcal O(p)=\ov{\mcal O (p,\psi^r)}.
\end{equation}
%------------------------------------------------------------
Since $p\prec q$,   
monotonicity
\[\big(\psi^r\big)^k p \preceq  \big (\psi^s\big)^k q= q,  \quad (k\in \Np),\]
whence
%============================================================
%% \begin{equation*}\label{eq:opq}
%============================================================
  \[\mcal O (p;\psi^r) \preceq \{q\},
  \]
%%\end{equation*}
%------------------------------------------------------------
and (\ref{eq:oppr}) implies
%============================================================
\begin{equation}\label{eq:opq}
%============================================================
\mcal O (p) \prec  \{q\}.
\end{equation}
Therefore  $\mcal O (p) \prec \mcal O(q)$ by monotonicity
(\ref{eq:gama}) and monotonicity.  But this  contradicts the hypothesis.
 \end{proof}
%============================================================
\begin{definition}\mylabel{th:resdef}
%============================================================
  A set $S$  is {\em resonant} for the flow  $(\psi,Y)$ if $S\subset Y$ and 
\[
  u,v\in S\cap(\P \sm \E) \implies
 \frac{\msf{per}(u)}{\msf{per}(v)}\in \Qp.
\]
%------------------------------------------------------------
\end{definition}
It is easy to prove:

%============================================================
\begin{lemma}\mylabel{th:resorbit}
%============================================================
\mbox{If $S$ is resonant, so is its orbit.  \qed}
\end{lemma}
%------------------------------------------------------------

%============================================================
\begin{proposition}\mylabel{th:resprop}
%============================================================
 $S$ is resonant provided there exists 
\[q \in S\cap (\P \sm \E )
\]
such that
%============================================================
\begin{equation}\label{eq:zsp}
%============================================================
z\in S\cap (\P  \sm \E) \implies 
  \frac{\msf{per} (z)}{\msf{per} (q)} \in \Qp.
\end{equation}
%------------------------------------------------------------
\end{proposition}
%--------------------------------------------------------------
\begin{proof} If 
  $ u,v\in S\cap\P (\psi)\sm\E$, then 
  \[
  \frac{\msf{per}(u)}{\msf{per}(v)} =
  \frac{\msf{per}(u)}{\msf{per}(q)}\cdot  \frac{\msf{per}(q)}{\msf{per}(v)},
  \]
which lies in $\Qp$ by (\ref{eq:zsp}).
\end{proof}
%------------------------------------------------------------

%============================================================
\begin{theorem} \mylabel{th:resglobal}
%===========================================================
Assume
%============================================================
  \begin{equation}\label{eq:pqP}
%============================================================
p, q \in \P \sm \E,      
\qquad p\prec q,   \qquad   \mcal O(p) \not \prec \mcal O(q).
\end{equation}
%------------------------------------------------------------
  Then $[p,q]$ is resonant.
\end{theorem}

%------------------------------------------------------------
\begin{proof} Note that 
%============================================================
  \begin{equation}\label{eq:pzq}
%============================================================
p\preceq  z \prec q \implies \mcal O(z)\not \prec \mcal O(q).
\end{equation}
%------------------------------------------------------------
For if this is false, there exists $z \in [p,q]$ such that
\[ \mcal O(z) \prec \mcal O(q),\]
whence  monotonicity implies
\[\big(\forall\, w\in\mcal O(p)\big) \ \big(\exists\, w'\in \mcal O(z)\big)\co\quad
w\preceq w'\prec \mcal O (q).
\]
It follows that  $\mcal O(p)\prec \mcal O(q)$, contrary to hypothesis.   
Resonance of $[p,q]$ now follows from Equation (\ref{eq:pzq}), and
Theorem \ref{th:poa} with the parameters $a:=p, \ b:=q$.
\end{proof}

The next result will be used to derive the Main  Theorem from the
analogous reuslt  for  homeomorphisms, Theorem \ref{th:lemmens}.

%============================================================
\begin{theorem}\mylabel{th:rescor}
%============================================================
Assume  $Y$ is resonant, $s >0$ and   $\P^s  \ne\varnothing$. Then $ \P (\psi)= \P (\psi^s)$.
\end{theorem}
%------------------------------------------------------------
\begin{proof}  %% Set $T:=\psi^s$.
 We  fix $p\in \P^s$ and show that every 
  $q\in \P (\psi)$ lies $\P (\psi^s)$.
  If then $q$ is stationary then $q\in \P (\psi^s)$.
  If  $q\in \P^r$ then  $r >0$ and resonance implies
$ sk = rl$ with $ k, l \in \Np.$

  Since $\psi^r q =q$, we have
  $ T^kq = \big(\psi^s\big)^kq=  \big(\psi^r\big)^lq =q$, and again
  $q\in \P (\psi^s)$.
\end{proof}

%------------------------------------------------------------
\section{Proof of Theorem \ref{th:main} }   \mylabel{sec:proofmain}
%------------------------------------------------------------
%% {\bf ***[FIX below here:] }

Recall the statement of the Theorem:

\smallskip
\noindent
{\em Assume:
  \begin{description}
 
  \item[(H1)] $K\subset \R n$ is an order cone,

  \item[(H2)] $X\subset \R n$ is open and connected,

  \item[(H3)] $X$ is $K$-ordered,

  \item[(H4)] the flow $(\v, X)$ monotone and densely periodic.
  \end{description}
  Then $\v$ is  periodic.
}

\medskip
The proof relies on two sufficient conditions for periodicity of 
monotone homeomorphisms,  Theorems  \ref{th:lemmens} and
\ref{th:monty} below. 

%============================================================
\begin{definitions*} %% \mylabel{th:}
%============================================================
A  homeomorphism $T\co W \approx W$ is:

\begin{itemize}

\item  {\em densely periodic}\,   if $\P(T)$ is dense in $W$,
  
\item  {\em pointwise periodic}\, if $\P(T)= W$,
  
\item  {\em  periodic}\,  if  $T^k=id_W$  for some $k \in \Np$.
\end{itemize}

\end{definitions*}
%------------------------------------------------------------

The key to Theorem \ref{th:main} is this striking result:
%==============================================================
\begin{theorem}[{\sc B. Lemmens} {\em et al.}, \cite {Lemmens17}]         \mylabel{th:lemmens}
%=============================================================
Assume {\em (H1), (H2), (H3)}. Then a monotone homeomorphism $T\co
X\approx X$ is periodic provided it is densely
periodic.\footnote{Conjectured, and proved for polyhedral cones, in
  {\sc M. Hirsch} \cite{Hirsch17}.}
\end{theorem}
We also use an elegant result from the early days of transformation groups:
%============================================================
\begin{theorem}[{\sc D. Montgomery}, \cite {Monty37, Monty55}] \mylabel{th:monty}
%==========================================================
  A homeomorphism of a connected topological manifold is
 periodic provided it is  pointwise
periodic.\footnote{There are  analogs of
  Montgomery's Theorem for countable transformation groups; see
  {\sc Kaul} \cite{Kaul71},
  {\sc Roberts} \cite {Roberts75},
  {\sc Yang} \cite {Yang71}.  Pointwise periodic homeomorphisms
  on compact metric spaces are investigated in {\sc Hall \& Schweigert} \cite
  {Hall82}.}
\end{theorem}
%---------------------------------------------------------

%============================================================
\begin{lemma}\mylabel{th:P}
%============================================================
$\v$ is pointwise periodic
\end{lemma}
%------------------------------------------------------------
%%{\bf (P)} \ {\em $\v$ is pointwise periodic.}
\begin{proof}
Since $\E\subset \P(\v)$, it suffices to prove that the restriction of
 $\varphi$ to each component of $X\,\verb=\=\E$ is pointwise periodic.
 Therefore we assume $\E=\empty$.
 
Let $x \in X$ be arbitrary.  As $\P(\v)$ is dense, Lemma
\ref{th:colimit} and Theorem \ref{th:resglobal} show there exist $p, q\in
\P(\v)$ such that the open set $[[p,q]]$ is connected and resonant, and
contains $x$.  Therefore the open set $Y:=\mcal O([[p,q]])$, which is 
invariant and connected, is resonant by Lemma \ref{th:resorbit}.

By Theorem \ref{th:rescor} there exists $s >0$ such that
$\P (\v^s) = \P (\varphi)\cap Y$. Consequently $\v^s$ is densely
periodic, hence periodic by Theorem \ref{th:lemmens}.  Therefore $x\in
\P(\v)$. \end{proof}
 Theorem \ref{th:main} follows from Lemma \ref{th:P} and 
Theorem \ref{th:monty}.

%-------------------------

%-------------------------------------------------------------------
\end{document}